\documentclass[11pt]{amsart}
\usepackage{xcolor}

\include{package}
\definecolor{grn}{rgb}{0,0.6,0}
\definecolor{mrn}{rgb}{0.3,0,0}
\definecolor{blue}{rgb}{0,0,0.7}
\definecolor{Mygray}{rgb}{0.75,0.75,0.75}
\definecolor{auburn}{rgb}{0.43, 0.21, 0.1}
\definecolor{britishracinggreen}{rgb}{0.0, 0.26, 0.15}
\definecolor{taupe}{rgb}{0.28, 0.24, 0.2}
\usepackage{amsmath,amssymb}
\newtheorem{theorem}{Theorem}[section]

\newtheorem{propn}{Proposition}[section]
\newtheorem{cor}{Corollary}[section]
\newtheorem{lemma}{Lemma}[section]

\newtheorem{quest}{Question}[section]
\newtheorem{rmk}{Remark}[section]

\newtheorem{conj}{Conjecture}[section]

\usepackage{latexsym,amssymb}
\usepackage{amssymb}
\usepackage{graphicx}
\usepackage{ textcomp }
\usepackage[left=2cm,right=2cm,top=2cm,bottom=2cm]{geometry}

\begin{document}
\baselineskip=14.5pt
\title[$p$-rationality of consecutive quadratic fields]{On the $p$-rationality of consecutive quadratic fields}

\author{Jaitra Chattopadhyay, H Laxmi and Anupam Saikia}
\address[Jaitra Chattopadhyay, H Laxmi and Anupam Saikia]{Department of Mathematics, Indian Institute of Technology Guwahati, Guwahati - 781039, Assam, India}
\email[Jaitra Chattopadhyay]{jaitra@iitg.ac.in; chat.jaitra@gmail.com}

\email[H Laxmi]{hlaxmi@iitg.ac.in}

\email[Anupam Saikia]{a.saikia@iitg.ac.in}

\begin{abstract}
In 2016, in the work related to Galois representations, Greenberg conjectured the existence of multi-quadratic $p$-rational number fields of degree $2^{t}$ for any odd prime number $p$ and any integer $t \geq 1$. Using the criteria provided by him to check $p$-rationality for abelian number fields, certain infinite families of quadratic, biquadratic and triquadratic $p$-rational fields have been shown to exist in recent years. In this article, for any integer $k \geq 1$, we build upon the existing work and prove the existence of infinitely many prime numbers $p$ for which the imaginary quadratic fields $\mathbb{Q}(\sqrt{-(p - 1)}),\ldots,\mathbb{Q}(\sqrt{-(p - k)})$ and $\mathbb{Q}(\sqrt{-p(p - 1)}),\ldots, \mathbb{Q}(\sqrt{-p(p - k)})$ are all $p$-rational. This can be construed as analogous results in the spirit of Iizuka's conjecture on the divisibility of class numbers of consecutive quadratic fields. We also address a similar question of $p$-rationality for two consecutive real quadratic fields by proving the existence of infinitely many $p$-rational fields of the form $\mathbb{Q}(\sqrt{p^{2} + 1})$ and $\mathbb{Q}(\sqrt{p^{2} + 2})$. The result for imaginary quadratic fields is accomplished by producing infinitely many primes for which the corresponding consecutive discriminants have large square divisors and the same for real quadratic fields is proven using a result of Heath-Brown on the density of square-free values of polynomials at prime arguments.
\end{abstract}

\renewcommand{\thefootnote}{}




\renewcommand{\thefootnote}{\arabic{footnote}}
\setcounter{footnote}{0}

\maketitle

\section{Introduction}

Let $p \geq 3$ be a prime number and $K$ be a number field. Let $M$ be the maximal $p$-ramified pro-$p$-extension of $K$ with Galois group ${\rm{Gal}}(M/K)$. The field $K$ is said to be $p$-rational if ${\rm{Gal}}(M/K)$ is a free pro-$p$-group. The notion of $p$-rationality was first introduced by Movaheddi (cf. \cite{mova-thesis}) in connection with the study of non-abelian number fields satisfying Leopoldt's conjecture. Interested readers are encouraged to look at \cite{assim}, \cite{gras1}, \cite{gras2}, \cite{gras3}, \cite{jaulent}, \cite{mova2} and \cite{mova3} for further details on this topic.

\smallskip

Recently, the study of $p$-rationality of number fields resurfaced in connection with the work of Greenberg \cite{greenberg} on the construction of Galois extensions of $\mathbb{Q}$ with Galois group being isomorphic to an open subgroup of $GL_{n}(\mathbb{Z}_{p})$. He proved that there exist continuous representations from the absolute Galois group ${\rm{Gal}}(\bar{\mathbb{Q}}/\mathbb{Q})$ of $\mathbb{Q}$ to the general linear group $GL_{n}(\mathbb{Z}_{p})$ for all integers $n \geq 4$, provided there exist $p$-rational multi-quadratic fields of arbitrarily large degree. He made a precise conjecture regarding the existence of such fields as follows.

\begin{conj} \cite{greenberg}\label{conjecture}
Let $p$ be an odd prime number and let $t \geq 1$ be an integer. Then there exists a $p$-rational number field whose Galois group over $\mathbb{Q}$ is isomorphic to $(\mathbb{Z}/2\mathbb{Z})^{t}$.
\end{conj}

Over the past few years, mathematicians have addressed Conjecture \ref{conjecture} for small values of $t$. In \cite{ray}, Barbulescu and Ray proved that the quadratic field $\mathbb{Q}(\sqrt{p^{2} - 1})$ is $p$-rational for all prime numbers $p$. In \cite{jpaa}, Benmerieme and Movahhedi proved that $\mathbb{Q}(\sqrt{p(p + 2)})$, $\mathbb{Q}(\sqrt{p(p - 2)})$ and the bi-quadratic field $\mathbb{Q}(\sqrt{p(p + 2)},\sqrt{p(p - 2)})$ are all $p$-rational for all prime $p \geq 5$. These results confirm Conjecture \ref{conjecture} for all odd primes $p$ and for $t =1 \mbox{ and } 2$. Recently, Koperecz addressed the case $t = 3$ and proved that the imaginary triquadratic field $\mathbb{Q}(\sqrt{p(p + 2)},\sqrt{p(p - 2)},\sqrt{-1})$ is $p$-rational for infinitely many primes $p$.

\smallskip

In this paper, we consider the question of $p$-rationality for certain {\it consecutive} imaginary as well as real quadratic fields. The motivation for such consideration comes essentially from a recent conjecture of Iizuka \cite{iizuka} and various works centering that. We state Iizuka's conjecture as follows
.

\begin{conj} \cite{iizuka}\label{iizuka-conj}
Let $\ell$ be a prime number and let $k \geq 1$ be an integer. Then there exist an infinite family of quadratic fields, real or imaginary, of the form $\mathbb{Q}(\sqrt{d}),\mathbb{Q}(\sqrt{d + 1}),\ldots,\mathbb{Q}(\sqrt{d + k})$ with $d \in \mathbb{Z}$ such that the class numbers of all of them are divisible by $\ell$.
\end{conj}

Iizuka \cite{iizuka} himself settled the conjecture for imaginary quadratic fields for $\ell = 3$ and $k = 1$. Recently, Conjecture \ref{iizuka-conj} has been settled for $k = 1$ and for all primes $\ell$ in \cite{tvm}. Some general cases have also been tackled in \cite{self-acta}, \cite{self-rama}, \cite{azizul} and \cite{rama-simul}. In a similar spirit as that of Conjecture \ref{iizuka-conj}, Cherubini et al. \cite{imrn} proved that there exist arbitrarily large number of consecutive real quadratic fields all of whose class numbers are larger than a given positive real number. In the light of Conjecture \ref{conjecture} and Conjecture \ref{iizuka-conj}, we ask the following question.

\begin{quest}\label{qn}
For any given integer $k \geq 2$, are there infinitely many primes $p$ for which there $k$ consecutive real or imaginary $p$-rational quadratic fields?
\end{quest}

We affirmatively answer Question \ref{qn} for imaginary quadratic fields by proving the following theorem.

\begin{theorem}\label{imaginary-case}
For any integer $k \geq 1$, there exist infinitely many primes $p$ such that the imaginary quadratic fields $\mathbb{Q}(\sqrt{-(p - 1)}),\ldots,\mathbb{Q}(\sqrt{-(p - k)})$ are simultaneously $p$-rational. Also, there exist infinitely many primes $p$ such that $\mathbb{Q}(\sqrt{-p(p - 1)}),\ldots, \mathbb{Q}(\sqrt{-p(p - k)})$ are all $p$-rational.
\end{theorem}


To address Question \ref{qn} for real quadratic fields, we prove the following theorem.

\begin{theorem}\label{real-case}
For sufficiently large prime number $p$, the real quadratic field $\mathbb{Q}(\sqrt{p^{2} + 1})$ is $p$-rational whenever $p^{2} + 1$ is square-free. The same holds true for $\mathbb{Q}(\sqrt{p^{2} - 2})$, $\mathbb{Q}(\sqrt{p^{2} + 2})$ and $\mathbb{Q}(\sqrt{p^{2} + 4})$ whenever the respective fundamental discriminants are square-free.
\end{theorem}

\begin{cor}\label{corollary}
There exist infinitely many prime numbers $p$ such that all the real quadratic fields of Theorem \ref{real-case} are simultaneously $p$-rational.
\end{cor}

\begin{rmk}
In \cite{ray}, Barbulescu and Ray proved the $p$-rationality of $\mathbb{Q}(\sqrt{p^{2} - 1})$ for all primes $p$. This, together with Theorem \ref{real-case} and Corollary \ref{corollary}, provides us with pairs of real quadratic fields of the form $(\mathbb{Q}(\sqrt{p^{2} - 2}),\mathbb{Q}(\sqrt{p^{2} - 1}))$ and $(\mathbb{Q}(\sqrt{p^{2} + 1}),\mathbb{Q}(\sqrt{p^{2} + 2}))$ that are $p$-rational for infinitely many primes $p$. This affirmatively answers Question \ref{qn} for real quadratic fields for $k = 2$.
\end{rmk}

\bigskip

\section{Preliminaries}

We begin with the following proposition due to Greenberg \cite{greenberg} to check for the $p$-rationality of abelian number fields.

\begin{propn} \cite[Proposition 3.6]{greenberg}\label{later-added}
Let $p$ be a prime number and let $K$ be an abelian number field such that the degree $[K : \mathbb{Q}]$ is indivisible by $p$. Then $K$ is $p$-rational if and only if every field $L$ with $\mathbb{Q} \subseteq L \subseteq K$ and $L/\mathbb{Q}$ cyclic is $p$-rational.
\end{propn} 

To deal with the $p$-rationality of quadratic fields, we recall the following criteria due to Greenberg \cite{greenberg}.

\begin{propn} \cite[Proposition 4.1]{greenberg}\label{greenberg-proposition}
Let $K$ be a quadratic field and let $p \geq 5$ be a prime number. 
\begin{enumerate}
\item If $K$ is real, then it is $p$-rational if and only if $p$ does not divide the class number $h_{K}$ of $K$ and the fundamental unit of $K$ is not a $p^{\rm{th}}$-power in the completion $K_{\mathfrak{p}}$ for some prime $\mathfrak{p}$ of $K$ lying above $p$.

\item If $K$ is imaginary, then $K$ is $p$-rational if and only if the Hilbert $p$-class field of $K$ is contained in the anti-cyclotomic $\mathbb{Z}_{p}$-extension of $K$. In particular, $K$ is $p$-rational if $p$ does not divide $h_{K}$.
\end{enumerate}
\end{propn}

In view of Proposition \ref{greenberg-proposition}, to prove Theorem \ref{imaginary-case}, it is sufficient to prove that the class numbers of all the fields of Theorem \ref{imaginary-case} are indivisible by $p$. We accomplish this by using Dirichlet's class number formula for the aforementioned fields and showing that all of their discriminants have large square factors, which is essentially a modification of the arguments used in \cite{kop}. We state a proposition from \cite{der} which will be used to produce infinitely many primes $p$ with $p - 1,\ldots, p - k$ simultaneously having large square factors.

\begin{propn}\cite[Proposition 1]{der}\label{der}
Let $m \geq 2$ be an integer. Then there exist a polynomial $f(X) = \displaystyle\prod_{i = 1}^{m}(a_{i}X + b_{i}) \in \mathbb{Z}[X]$ such that $\gcd(b_{i}, b_{j}) = 1 = \gcd(a_{i}k + b_{i}, a_{j}k + b_{j})$ for all $k \in \mathbb{Z}$ and for all $i, j \in \{1,\ldots,m\}$ with $i \neq j$.
\end{propn}

Now, we prove the following proposition which plays a crucial role in the proof of Theorem \ref{imaginary-case}. This is a generalization of Proposition 4 of \cite{kop}.

\begin{propn}\label{key-proposition}
For a given real number $A > 0$ and any given finitely many integers $r_{1},\ldots ,r_{s}$, there exist infinitely many prime numbers $p$ such that $p - r_{i}$ has a square factor larger than $(\log p)^{A}$ for each $i \in \{1,\ldots ,s\}$.
\end{propn}

\begin{proof}
Let $r = \displaystyle\prod_{i = 1}^{s}r_{i}$ and let $\mathcal{P}$ be the set of all prime numbers dividing $r$. For an arbitrary but fixed positive integer $m$, by Proposition \ref{der}, there exist polynomials $f_{i}(X) = (a_{i}X + b_{i}) \in \mathbb{Z}[X]$ for each $i \in \{1,\ldots,m\}$ such that $\gcd(b_{i},b_{j}) = 1$ and $\gcd(a_{i}k + b_{i},a_{j}k + b_{j}) = 1$ for all $k \in \mathbb{Z}$ and $i \neq j$. Let $g_{i}(X) := f_{i}(rX) = a_{i}rX + b_{i}$. Then for any $k \in \mathbb{Z}$, we have $$\gcd(g_{i}(k),g_{j}(k)) = \gcd(f_{i}(rk),f_{j}(rk)) = 1 \mbox{ whenever } i \neq j.$$ 

Now, we notice that for any integer $k$, the fact $\gcd(g_{i}(k), r) = 1$ is equivalent to $\gcd(b_{i}, r) = 1$. We call $g_{i}(X)$ {\it admissible} if $\gcd(b_{i},r) = 1$. Since $\gcd(b_{i},b_{j}) = 1$ for $i \neq j$, we conclude that each $p \in \mathcal{P}$ divides $b_{i}$ for at most one $i$. Consequently, there can be at most $\# \mathcal{P}$ many $i$ for which $g_{i}(X)$ fails to be admissible. Since $m$ is arbitrary, we can discard all the $g_{i}(X)$ that fail to be admissible and therefore the collection of the remaining $g_{i}(X)$ is admissible. In other words, we can find arbitrarily many finite number of linear polynomials $a_{i}X + b_{i} \in \mathbb{Z}[X]$ such that $\gcd(a_{i}k + b_{i},a_{j}k + b_{j}) = 1$ for $i \neq j$ and $\gcd(a_{i}k + b_{i},r) = 1$ for all $k \in \mathbb{Z}$.

\smallskip

Now, for sufficiently large $X$, we can choose integers $m_{1},\ldots ,m_{s}$ such that $\frac{1}{v}(\log X)^{A} \leq m_{i} \leq (\log X)^{A},$ $\gcd(m_{i},r) = 1$ for all $i \in \{1,\ldots ,s\}$ and $\gcd(m_{i},m_{j}) = 1$ for all $i \neq j$. Now, by Chinese remainder theorem, the system of congruence 

\begin{equation}\label{series}
\left\{\begin{array}{ll}
x \equiv r_{1} \pmod {m_{1}^{2}};\\

\vdots\\

x \equiv r_{k} \pmod {m_{s}^{2}}
\end{array}\right.
\end{equation}
has a unique solution $\ell \pmod {D}$, where $D = \displaystyle\prod_{i = 1}^{s}m_{i}^{2} \leq (\log X)^{c}$.

\smallskip

Now, we proceed as in \cite{kop}. For large positive real number $X$ and positive integers $D$ and $\ell$ with $\gcd(D,\ell) = 1$, let $\pi(X,D,\ell) := \{p \in \mathbb{N} : p \geq 2 \mbox{ is prime and } p \equiv \ell \pmod {D}\}$. Then for a fixed real number $c > 0$, the estimate $$\pi(X,D,\ell) = \frac{1}{\phi(D)} \displaystyle\int_{2}^{X}\frac{dt}{\log t} + O(Xe^{-c_{1}\sqrt{\log X}})$$ holds uniformly for all integers $D \mbox{ and } \ell$ and $1 \leq D \leq (\log X)^{c}$ \cite[Lemma 2.9]{elliot}.

\smallskip

Let $A > 0$ and let $c = 2sA + 1$. From the inequality $\pi(X,D,\ell) - \pi(\frac{X}{v},D,\ell) > 0$ for an arbitrary but fixed integer $v \geq 2$, we conclude that there exist a prime $p \equiv \ell \pmod {D}$ with $\frac{X}{v} < p < X$ for sufficiently large $X$.

\smallskip

\smallskip

Now, we choose $X$ suitably large enough so that there exists a prime number $p \in (\frac{X}{v},X)$ and $p \equiv \ell \pmod {D}.$ Then $p \equiv r_{i} \pmod {m_{i}^{2}}$ for all $i \in \{1,\ldots ,s\}$. This completes the proof of the proposition.
\end{proof}

One way to prove that the class number $h_{K}$ of the imaginary quadratic field $K$ is indivisible by a prime $p$ is to show that $h_{K} < p$. This motivates us to look for suitable upper bounds for the class numbers of imaginary quadratic fields and the following proposition of Louboutin \cite{loub} serves the desired purpose.

\begin{propn} \cite[Proposition 2]{loub}\label{louboutin-proposition}
For an imaginary quadratic field $K$ with discriminant $d_{K}$ and class number $h_{K}$, we have $$h_{K} \leq \frac{\omega_{K} \cdot \sqrt{|d_{K}|}}{4\pi}\left(\log |d_{K}| + \frac{3}{2}\right),$$ where $\omega_{K}$ stands for the number of roots of unity in $K$.
\end{propn}

The next proposition is due to Heath-Brown \cite{heath-brown} and is about the square-free values of an integral polynomial at prime arguments. This plays an important role in the proof of Corollary \ref{corollary} because it requires us to consider the simultaneous square-free values of certain quadratic polynomials. We state it as follows.

\begin{propn} \cite[Theorem 1.2]{heath-brown}\label{square-free proposition}
Let $f(X) = X^{d} + c \in \mathbb{Z}[X]$ be irreducible and let $k \geq \frac{5d + 3}{9}$ be an integer. Suppose that for every prime number $p$, there exists an integer $n_{p}$ with $\gcd(p,n_{p}) = 1$ and $f(n_{p}) \not\equiv 0 \pmod{p^{k}}$. For a positive real number $X$, let $N^{\prime}_{f,k}(X) := \{p \mbox{ prime } : p \leq X \mbox{ and } f(p) \mbox{ is } k\mbox{-free}\}$. Then for any fixed $A > 0$, the estimate 

\begin{equation}\label{heath-brown estimate}
N^{\prime}_{f,k}(X) = \displaystyle\prod_{p}\left(1 - \frac{\rho^{\prime}_{f}(p^{k})}{\phi(p^{k})}\right)\pi (X) + O_{A}\left(\frac{X}{(\log X)^{A}}\right)
\end{equation}
holds and the implied constant depends on $A$. Here $\pi (X)$ stands for the number of primes up to $X$ and $\rho^{\prime}_{f}(d) := \#\{n \pmod {d} : \gcd(n,d) = 1 \mbox{ and } f(n) \equiv 0 \pmod {d}\}$.
\end{propn}

\begin{rmk}
In Proposition \ref{square-free proposition}, we immediately see that the hypotheses are satisfied for the particular choice $d = k = 2$ and $c = -2, 1, 2$ and $4$. In that case, $f$ is a quadratic polynomial and hence $\rho^{\prime}_{f}(p^{2}) = 0 \mbox{ or } 2$, depending on the appropriate congruence class of $p$ modulo $8$. In other words, $\rho^{\prime}_{f}(p^{2})$ is an absolute constant and $\phi(p^{2}) = p^{2} - p$ is of the order of $p^{2}$ for large enough $p$. Thus the Euler product in \eqref{heath-brown estimate} converges to a non-zero constant and hence the square-free values of the respective quadratic polynomials indeed have positive relative density in the set of prime numbers.
\end{rmk}

In view of Proposition \ref{greenberg-proposition}, to establish the $p$-rationality of a real quadartic field $K$ in which $p$ is unramified, it is required to prove that the fundamental unit is not a $p^{\rm{th}}$-power in the completion $K_{\mathfrak{p}}$ for some prime $\mathfrak{p}$ of $K$ lying above $p$. Let $\varepsilon$ be the fundamental unit of $K$ and let $q: = p^f$, where $f$ is the residual degree of $p$ in $K$. Then $f$ either 1 or 2, depending on whether $p$ is an inert or a split prime. The unit $\varepsilon^{q-1}$ is a principal unit (that is, it belongs to $U_{\mathfrak{p}}^{(1)} := 1 + \mathfrak{p}\mathcal{O}_{\mathfrak{p}}$) in the completion $K_{\mathfrak{p}}$ and is locally a $p$-th power if and only if it belongs to $U_{\mathfrak{p}}^{(2)} := 1 + \mathfrak{p}^{2}\mathcal{O}_{\mathfrak{p}}$ (cf. \cite[Proposition 9, Chapter 14]{serre}). Let $\mathfrak{p}\mathcal{O}_{\mathfrak{p}}$ be the maximal ideal of $K_{\mathfrak{p}}$. Then, $\varepsilon$ is locally a $p$-th power if and only if when $\varepsilon^{q-1} \equiv 1 \pmod{\mathfrak{p}^{2}\mathcal{O}_{\mathfrak{p}}}$. Equivalently, $\varepsilon^{q} \equiv \varepsilon \pmod{\mathfrak{p}^{2}\mathcal{O}_{\mathfrak{p}}}.$ Also,
$$\varepsilon^{q-1} \equiv 1 \pmod{\mathfrak{p}^{2}\mathcal{O}_{\mathfrak{p}}} \Leftrightarrow \varepsilon^{1-q} \equiv 1 \pmod{\mathfrak{p}^{2}\mathcal{O}_{\mathfrak{p}}} \Leftrightarrow {(\varepsilon^{-1}})^{q} \equiv \varepsilon^{-1}\ \pmod{\mathfrak{p}^{2}\mathcal{O}_{\mathfrak{p}}}.$$ Consequently, we get that $\varepsilon^q - ({\varepsilon^{-1}})^{q} \equiv  \varepsilon-\varepsilon^{-1}\ \pmod{\mathfrak{p}^{2}\mathcal{O}_{\mathfrak{p}}}$, which, in turn, implies that $\dfrac{\varepsilon^q - ({\varepsilon^{-1}})^{q}}{\varepsilon-\varepsilon^{-1}} \equiv 1 \pmod{\mathfrak{p}^{2}\mathcal{O}_{\mathfrak{p}}}.$

\smallskip

Now, we define the {\it generalized Fibonacci number} by the recursion formula $\mathcal{F}_n := \dfrac{\varepsilon^n - ({\varepsilon^{-1}})^{n}}{\varepsilon-\varepsilon^{-1}}$ with the initial conditions $\mathcal{F}_0 = 0$ and $\mathcal{F}_1 = 1$. We note that $\varepsilon^{-1}$ can be replaced by $\bar\varepsilon$ whenever they are equal. The next lemma provides a necessary and sufficient condition for a real quadratic field to be $p$-rational in terms of the generalized Fibonacci number. 

\begin{lemma} (cf. \cite[Corollary 3.1]{jpaa})\label{Fibonacci-lemma}
Let $F$ be a real quadratic field in which the odd prime $p$ does not ramify. Then $F$ is $p$-rational if and only if the following conditions hold.
\begin{enumerate}
\item The class number $h_{F}$ of $F$ is indivisible by $p$.
\item The congruence $\mathcal{F}_q \not\equiv 1 \pmod{p^{2}}$ holds for the generalized Fibonacci number.
\end{enumerate}
\end{lemma} 

\section{Proof of Theorem \ref{imaginary-case}}

By Proposition \ref{key-proposition}, there exist infinitely many primes $p$ such that $p - j$ has a divisor $\ell^{2}$ such that $\ell > (\log p)^{2}$ for each $j \in \{1,\ldots,k\}$. Let $K_{j} := \mathbb{Q}(\sqrt{-(p - j)})$. Then the discriminant $d_{K_{j}}$ of $K_{j}$ satisfies the inequality $$|d_{K_{j}}| \leq 4 \times \mbox {square-free part of } (p - j) \leq \frac{4(p - j)}{(\log p)^{4}}.$$ Therefore, by using Proposition \ref{louboutin-proposition}, we obtain 

\begin{equation}\label{K_{j}}
h_{K_{j}} \leq \frac{\omega_{K_{j}}}{4\pi}\sqrt{\frac{4(p - j)}{(\log p)^{4}}}\left(\log \left(\frac{4(p - j)}{(\log p)^{4}}\right) + \frac{3}{2}\right).
\end{equation}
Since $\omega_{K_{j}} = 2,4$ or $6$, we conclude from \eqref{K_{j}} that $h_{K_{j}} \ll \frac{\sqrt{p}}{\log p} \ll p$. Consequently, $h_{K_{j}} < p$ for sufficiently large primes $p$ and therefore, $p$ does not divide $h_{K_{j}}$. By Proposition \ref{greenberg-proposition}, we conclude that each $K_{j}$ is $p$-rational. 

\smallskip

Similarly, we let $F_{j} := \mathbb{Q}(\sqrt{-p(p - j)})$. Then the discriminant $d_{F_{j}}$ of $F_{j}$ satisfies the inequality $$|d_{F_{j}}| \leq 4\times \mbox {square-free part of } p(p - j) \leq \frac{4p(p - j)}{(\log p)^{4}}.$$ Therefore, by using Proposition \ref{louboutin-proposition}, we obtain 

\begin{equation}\label{F_{j}}
h_{F_{j}} \leq \frac{\omega_{F_{j}}}{4\pi}\sqrt{\frac{4p(p - j)}{(\log p)^{4}}}\left(\log \left(\frac{4p(p - j)}{(\log p)^{4}}\right) + \frac{3}{2}\right) \ll \frac{\sqrt{p(p - j)}}{\log p} \ll p.
\end{equation}

Hence $p$ does not divide $h_{F_{j}}$ and thus by Proposition \ref{greenberg-proposition}, we conclude that $F_{j}$ is $p$-rational. This completes the proof of Theorem \ref{imaginary-case}. $\hfill\Box$

\section{Proof of Theorem \ref{real-case}}

We give a complete proof for the field $K = \mathbb{Q}(\sqrt{p^{2} + 1})$, assuming that $p^{2} + 1$ is square-free. The proofs for other three fields follow similar line of argument.

By using Lemma \ref{Fibonacci-lemma}, we first prove that the fundamental unit $\varepsilon$ is not a $p^{\rm{th}}$-power in $K_{\mathfrak{p}}$. Since $p$ is odd, $p^2 + 1 \equiv 2 \pmod{4}$. Since $p^2 + 1$ is square free, we have that $d_K = 4(p^2 +1).$ Also, $\varepsilon = p + \sqrt{p^2 +1}$ is a fundamental unit of $K$ (cf. \cite[Chapter 8]{ram}). Let $q \in \mathbb{N}$ be odd. Using the binomial theorem, we have $\varepsilon^q = (p + \sqrt{p^2 +1})^q =(\sqrt{p^2 +1})^q + qp(\sqrt{p^2 +1})^{q-1} + \cdots + qp^{q-1}(\sqrt{p^2 +1})+ p^q$. Similarly, $({\varepsilon^{-1}})^q = (-p + \sqrt{p^2 +1})^q = (\sqrt{p^2 +1})^q - qp(\sqrt{p^2 +1})^{q-1} + \cdots + qp^{q-1}(\sqrt{p^2 +1})- p^q.$\\
Therefore, the $q^{\rm{th}}$ generalized Fibonacci number is 
\begin{align*}
\mathcal{F}_q &= \dfrac{\varepsilon^q - ({\varepsilon^{-1}})^q }{\varepsilon - \varepsilon^{-1}}\\
 &= \dfrac{2qp(\sqrt{p^2 +1})^{q-1} + \cdots + 2\frac{q(q-1)}{2}p^{q-2}(\sqrt{p^2 +1})^{2} + 2p^q}{2p}\\
 &= q(p^2 +1)^{\frac{q-1}{2}} + \cdots + \frac{q(q-1)}{2}p^{q-3}(p^2 +1) + p^{q-1}.
\end{align*}
If $q =p,$ then $\mathcal{F}_q = p(p^2 +1)^{\frac{p-1}{2}} + \cdots + \frac{p(p-1)}{2}p^{p-3}(p^2 +1) + p^{p-1}.$ Since $p-1$ is even, taking congruence $\pmod {p^2}$, we get
$$\mathcal{F}_q \equiv p(p^2 +1)^{\frac{p-1}{2}} \equiv p\ \not\equiv 1 \pmod {p^{2}}.$$

If $q = p^2$, then $\mathcal{F}_q \equiv 0\ \not\equiv 1 \pmod {p^{2}}$. Thus, in either case, the $q^{\rm{th}}$ generalized Fibonacci number $\mathcal{F}_q \not\equiv 1 \pmod {p^{2}}$. Consequently, by Lemma \ref{Fibonacci-lemma}, the fundamental unit $\varepsilon$ is not a $p^{\rm{th}}$-power in the completion $K_{\mathfrak{p}}$. 

\smallskip
 
Next, we prove that the class number $h_{K}$ is not divisible by $p$, for sufficiently large primes $p$. From Dirichlet's class number formula, we have $h_K =\dfrac{ L(1, \chi_F)\sqrt{d_K}}{2R_K}$, where $R_K := \log(\varepsilon)$ is the regulator of $K$. Since $2$ is ramified in $K$, it further leads to the following inequality \cite[Corollary 2]{loub} $$ L(1, \chi_K) \leq \dfrac{\log d_K + \kappa_2}{4},$$ 
where $\kappa_2 := 2 + \gamma - \log(\pi) \sim 1.432\ldots$, and $\gamma$ is the Euler's constant. Using $R_K = \log(p + \sqrt{p^2 +1})$, we get the following inequality

\begin{eqnarray*}
h_K &<& \dfrac{(\log d_K +2)}{4}\times \dfrac{\sqrt{d_K}}{2R_K}\\
  &=& \dfrac{\log(4 ( p^2 +1)) +2}{4}\times \dfrac{\sqrt{4(p^2 +1)}}{2\log(p + \sqrt{p^2 +1})}\\
  &=& \dfrac{\log(4 ( p^2 +1)) +2}{2}\times \dfrac{\sqrt{p^2 +1}}{\log((p + \sqrt{p^2 +1})^2)}\\
  &=& \dfrac{\left( \log(4 ( p^2 +1)) +2 \right)\sqrt{p^2+1}}{2\log(p^2 + p^2 + 1 + 2p\sqrt{p^2+1})}\\
  &=& \left( \dfrac{\log(4 ( p^2 +1))}{2\log(2p^2 + 1 + 2p\sqrt{p^2+1})} + \dfrac{1}{\log(2p^2 + 1 + 2p\sqrt{p^2+1})} \right)\sqrt{p^2 +1}.
\end{eqnarray*}
Since $\sqrt{p^2 +1} > p$, we obtain
\begin{align*}
2\log(2p^2 + 1 + 2p\sqrt{p^2+1}) &> 2\log(2p^2 + 1 + 2p^2) \\
 &= 2\log(4p^2 + 1) = \log( (4p^2 + 1)^2)\\
 &= \log( 16p^4 + 8p^2 + 1).
\end{align*}
Therefore, $h_K < \left( \dfrac{\log(4 ( p^2 +1))}{\log(16p^4 + 8p^2 + 1)} + \dfrac{1}{\log(2p^2 + 1 + 2p\sqrt{p^2+1})} \right)\sqrt{p^2 +1}.$\\
For sufficiently large $p$, the quantity inside the bracket becomes smaller than $1$ and consequently, $h_K < p$, implying that $p$ does not divide $h_K$  for sufficiently large primes $p$. Therefore, $K = \mathbb{Q}(\sqrt{p^2 + 1})$ is $p$-rational whenever $p^2 +1$ is square-free. Since by Proposition \ref{heath-brown estimate}, there exist infinitely many primes $p$ for which $p^{2} + 1$ is square-free, we obtain the $p$-rationality for infinitely many such fields. $\hfill\Box$

\section{Proof of Corollary \ref{corollary}}

In the light of Theorem \ref{real-case}, it suffices to prove that there exist infinitely many prime numbers $p$, with positive lower relative density, such that all the integers $p^{2} - 2, p^{2} + 1, p^{2} + 2$ and $p^{2} + 4$ are simultaneously square-free. Let $$A_{1}:= \{p \mbox{ prime } : p^{2} + 1 \mbox{ is square-free }\},$$ $$A_{2} := \{p \mbox{ prime } : p^{2} -2  \mbox{ is square-free }\},$$ $$A_{3} := \{p \mbox{ prime } : p^{2} + 2 \mbox{ is square-free }\},$$ $$A_{4} := \{p \mbox{ prime } : p^{2} + 4 \mbox{ is square-free }\}.$$ For a positive real number $X$, let $A_{i}(X) := \{p \in A_{i} : p \leq X\}$, for each $i = 1,2,3$ and $4$. We prove that the set $\displaystyle\bigcap_{i = 1}^{4} A_{i}$ has a positive lower relative density.

\smallskip

For each $i \in \{1,2,3,4\}$, let $\delta(A_{i})$ denote the relative density of $A_{i}$ in the set of all prime numbers. That is, $\delta(A_{i}) := \displaystyle\lim_{X \to \infty} \frac{\#A_{i}(X)}{\pi(X)}$. Then by Proposition \ref{square-free proposition}, we conclude that $\delta(A_{i})$ exists for each $i$. Now, for $A_{1}$, by Proposition \ref{square-free proposition}, we have $$\delta(A_{1}) = \displaystyle\prod_{p \equiv 1 \pmod {4}}\left(1 - \frac{2}{p(p - 1)}\right) > \displaystyle\prod_{k = 1}^{\infty}\left(1 - \frac{2}{4k(4k + 1)}\right) \sim 0.834.$$ Similarly, we obtain $\delta(A_{2}) \geq 0.931$, $\delta(A_{3}) \geq 0.920$ and $\delta(A_{4}) \geq 0.834$. Now, we obtain
 
\begin{eqnarray*}
\displaystyle\liminf_{X \to \infty} \frac{\#(A_{1}(X) \cap A_{2}(X))}{\pi(X)} &\geq &  \displaystyle\liminf_{X \to \infty} \frac{\#A_{1}(X)}{\pi(X)}+ \displaystyle\liminf_{X \to \infty} \frac{\#A_{2}(X)}{\pi(X)} - \displaystyle\limsup_{X\to \infty} \frac{\#(A_{1}(X) \cup A_{2}(X))}{\pi(X)}\\ & \geq & 0.834 + 0.931 - 1\\ & = & 0.765.
\end{eqnarray*}


Similarly, proceeding as above, we obtain $\displaystyle\liminf_{X \to \infty} \dfrac{\#((A_{1}(X)\cap A_{2}(X)) \cap A_{3}(X))}{\pi(X)} \geq 0.685$. Finally, using the above estimate, we obtain $\displaystyle\liminf_{X \to \infty} \dfrac{\#\left(\displaystyle\bigcap_{i = 1}^{4}A_{i}(X)\right)}{\pi(X)} \geq 0.519$. This completes the proof of the corollary. $\hfill\Box$


\section{Concluding remarks}

Towards the end of \cite{jpaa}, the totally real bi-quadratic field $K_{\alpha} := \mathbb{Q}(\sqrt{\alpha p(\alpha p + 2)},\sqrt{\alpha p (\alpha p - 2)})$ has been mentioned, where $p > 3$ is a prime number and $\alpha$ is a postive integer with $\gcd(\alpha, p) = 1$. It has been explicitly written (at page number $15$ of \cite{jpaa}) that ``{\it ... the fundamental unit of each subfield of $K_{\alpha}$ is not locally a $p^{{\rm{th}}}$ power at the $p$-adic places. So $K_{\alpha}$ is $p$-rational as soon as $p$ does not divide $h_{K_{\alpha}}$.Though there exist $\alpha$ such that $p \mid h_{K_{\alpha}}$ (for instance with $p = 5$ and $\alpha = 17$), it would be interesting to find an infinite family of integers $\alpha$ for which $p$ does not divide $h_{K_{\alpha}}$.}"

\smallskip

Here, we make an attempt to address this problem by using Proposition \ref{key-proposition}. We notice that the quadratic subfields of $K_{\alpha}$ are precisely $K_{1} := \mathbb{Q}(\sqrt{\alpha p(\alpha p + 2)})$, $K_{2} := \mathbb{Q}(\sqrt{\alpha p(\alpha p - 2)})$ and $K_{3} := \mathbb{Q}(\sqrt{(\alpha p - 2)(\alpha p + 2)})$. Then as in the system of congruence in \eqref{series}, we may consider 

\begin{equation}\label{series-new}
\left\{\begin{array}{ll}
x \equiv 2\alpha^{-1} \pmod {m^{2}}\\
x \equiv -2\alpha^{-1} \pmod {n^{2}}
\end{array}\right.
\end{equation}
where $m$ and $n$ are integers suitably chosen in certain range of $\log X$ such that $\gcd(m,n) = \gcd(m,\alpha) = \gcd(n,\alpha) =  1$. This choice is possible because of Proposition \ref{key-proposition}. Dirichlet's theorem for primes in arithmetic progressions asserts that there exist infinitely many prime numbers $p$ satisfying the system of congruence \eqref{series-new}. In other words, there exist infinitely many prime numbers $p$ such that both $\alpha p - 2$ and $\alpha p + 2$ have square divisors $> (\log p)^{A}$ for arbitrary but fixed constant $A > 0$.

\smallskip

Hence for those choices of prime numbers $p$ and integers $\alpha \ll \log \log p$, using Le's bound for class numbers of real quadratic fields (cf. \cite[Theorem (a)]{le}), we have $$h_{K_{1}} \leq \frac{1}{2}\sqrt{d_{K_{1}}}  \leq \frac{1}{2}\sqrt{\frac{\alpha p(\alpha p + 2)}{(\log p)^{4}}} \ll \frac{p\log \log p}{(\log p)^{2}} \ll p.$$ Consequently, for sufficiently large such prime numbers $p$, we conclude that $p$ does not divide $h_{K_{1}}$. Similarly, for $h_{K_{2}}$ and $h_{K_{3}}$ also we arrive at the same conclusion. Thus we proved the following proposition.

\begin{propn}
There exist infinitely many prime numbers $p$ and for each such $p$, there exist finitely many positive integers $\alpha$ with $\gcd(\alpha,p) = 1$ such that each $K_{\alpha}$ is $p$-rational.
\end{propn}

Moreover, by the concluding remarks in \cite{kop}, we obtain the following corollary.

\begin{cor}
There exist infinitely many prime numbers $p$ such that the tri-quadratic field $K_{\alpha}(\sqrt{-1})$ is $p$-rational.
\end{cor}

By Proposition 4.4 of \cite{jpaa}, we know that the bi-quadratic field $\mathbb{Q}(\sqrt{p(p + 2)},\sqrt{p(p - 2)})$ is $p$-rational for all primes $p$. From this, another interesting thing to observe is the following proposition, which is in a similar spirit as that of Koperecz in \cite{kop}.

\begin{propn}
Let $\alpha \geq 1$ be an integer. Then there exist infinitely many primes $p$ such that the tri-quadratic field $F_{\alpha} := \mathbb{Q}(\sqrt{p(p + 2)},\sqrt{p(p - 2)},\sqrt{-\alpha})$ is $p$-rational.
\end{propn}

\begin{proof}
In view of Proposition \ref{later-added}, it suffices to establish the $p$-rationality of the quadratic subfields of $F_{\alpha}$. Since the $p$-rationality is known for $\mathbb{Q}(\sqrt{p(p + 2)},\sqrt{p(p - 2)})$ for all primes $p$, we only need to check for the $p$-rationality of the imaginary quadratic fields $\mathbb{Q}(\sqrt{-\alpha})$, $\mathbb{Q}(\sqrt{-p\alpha(p + 2)})$, $\mathbb{Q}(\sqrt{-p\alpha(p - 2)})$ and $\mathbb{Q}(\sqrt{-\alpha(p - 2)(p + 2)})$. By Proposition \ref{key-proposition}, there exist infinitely many primes $p$ such that $p - 2$ and $p + 2$ simultaneously have large square divisors. Therefore, proceeding as in the proof of Theorem \ref{imaginary-case}, we obtain that $p$ does not divide the respective class numbers for sufficiently large primes $p$. Since $\alpha$ is a fixed integer, we can choose $p > h_{\mathbb{Q}(\sqrt{-\alpha})}$ and hence $h_{\mathbb{Q}(\sqrt{-\alpha})}$ is indivisible by $p$. Thus all the quadratic subfields of $F_{\alpha}$ are $p$-rational. Consequently, $F_{\alpha}$ is $p$-rational. This completes the proof of the proposition. 
\end{proof}

We furnish some values of the class numbers of the imaginary and real quadratic fields considered in Theorem \ref{imaginary-case} and Theorem \ref{real-case} for certain values of the prime $p$. In the following tables, we use the notation $h(d)$ to denote the class number of $\mathbb{Q}(\sqrt{d})$. The computations of the class numbers have been carried out using MAGMA.

\begin{table}[h!]
\begin{center}
\caption{Simultaneous $p$-rationality of $\mathbb{Q}(\sqrt{-(p - 1)}),\ldots, \mathbb{Q}(\sqrt{-(p - 5)})$.}
\smallskip
\label{tab:table7}
\begin{tabular}{|c|c|c|c|c|c|} 
\hline 
$p$ & $h(-(p - 1))$ & $h(-(p - 2))$ & $h(-(p - 3))$ & $h(-(p - 4))$ & $h(-(p - 5))$ \\
\hline
$23$ & $2$ & $4$ &  $2$ & $1$ & $1$ \\
\hline
$29$ & $1$ & $1$  & $6$ & $1$ & $2$ \\
\hline
$31$ & $4$ & $6$ & $1$ & $1$ & $6$ \\
\hline
$37$ & $1$ & $2$  & $4$ & $4$ & $1$ \\
\hline
$41$ & $2$ & $4$ & $6$ & $2$ & $1$ \\
\hline
$43$ & $4$ & $8$ & $2$ & $4$ & $6$ \\
\hline
$47$ & $4$ & $2$ & $1$ & $1$ & $4$ \\
\hline
$53$ & $2$ & $2$ & $1$ & $1$ & $1$ \\
\hline
$59$ & $2$ & $4$ & $4$ & $4$ & $2$ \\
\hline
$61$ & $2$ & $3$ & $2$ & $4$ & $4$ \\
\hline

\end{tabular}
\end{center}
\end{table}

\begin{table}[h!]
\begin{center}
\caption{Simultaneous $p$-rationality of $\mathbb{Q}(\sqrt{-p(p - 1)}),\ldots, \mathbb{Q}(\sqrt{-p(p - 5)})$.}
\smallskip
\label{tab:table8}
\begin{tabular}{|c|c|c|c|c|c|} 
\hline 
$p$ & $h(-p(p - 1))$ & $h(-p(p - 2))$ & $h(-p(p - 3))$ & $h(-p(p - 4))$ & $h(-p(p - 5))$ \\
\hline
$7$ & $4$ & $2$ & $1$ & $4$ & $4$ \\
\hline
$13$ & $4$ & $10$ & $4$ & $2$ & $6$ \\
\hline
$17$ & $4$ & $12$ &  $8$ & $16$ & $2$ \\
\hline
$29$ & $4$ & $6$  & $20$ & $6$ & $12$ \\
\hline
$31$ & $24$ & $14$ &  $8$ & $4$ & $28$ \\
\hline
$37$ & $2$ & $36$  & $12$ & $32$ & $10$ \\
\hline
$53$ & $40$ & $28$ & $6$ & $6$ & $10$ \\
\hline
$59$ & $52$ & $16$  & $12$ & $48$ & $16$ \\
\hline
$71$ & $56$ & $16$ & $18$ & $52$ & $56$ \\
\hline
$79$ & $40$ & $24$ & $24$ & $12$ & $76$ \\
\hline

\end{tabular}
\end{center}
\end{table}

\begin{table}[h!]
\begin{center}
\caption{Simultaneous $p$-rationality of $\mathbb{Q}(\sqrt{p^2 +1})$, $\mathbb{Q}(\sqrt{p^2 - 2})$, $\mathbb{Q}(\sqrt{p^2 + 2})$ and $\mathbb{Q}(\sqrt{p^2 + 4})$.}
\smallskip
\label{tab:table1}
\begin{tabular}{|c|c|c|c|c|} 
\hline 
$p$ & $p^2 +1 $ & $p^2 - 2$ & $p^2 + 2$ & $p^2 + 4$ \\
\hline
$3$ & $10$ & $7$ & $11$ & $13$ \\
\hline
$17$ & $290$ & $287$ & $291$ & $293$ \\
\hline
$37$ & $1370$ & $1367$ & $1371$ & $1373$ \\
\hline
$47$ & $2210$ & $2207$ & $2211$ & $2213$ \\
\hline
$53$ & $2810$ & $2807$ & $2811$ & $2813$ \\
\hline
$73$ & $5330$ & $5327$ & $5331$ & $5333$ \\
\hline
$79$ & $6242$ & $6239$ & $6243$ & $6245$ \\
\hline
$83$ & $6890$ & $6887$ & $6891$ & $6893$ \\
\hline
$97$ & $9410$ & $9407$ & $9411$ & $9413$ \\
\hline
\end{tabular}
\end{center}
\end{table}

\medskip

\bigskip

{\bf Acknowledgements.} We would like to thank IIT Guwahati for providing excellent facilities to carry out the research. The first author gratefully acknowledges the National Board of Higher Mathematics (NBHM) for the Post-Doctoral Fellowship (Order
No. 0204/16(12)/2020/R \& D-II/10925). The third author thanks MATRICS, SERB for the research grant MTR/2020/000467.


\begin{thebibliography}{9999}

\bibitem{der}
Anand, J. Chattopadhyay and B. Roy, {\it On sums of polynomial-type exceptional units in $\mathbb{Z}/n\mathbb{Z}$}, Arch. Math. (Basel), {\bf 114} (2020), 271-283.

\bibitem{assim}
J. Assim and Z. Bouazzaoui, {\it Half-integral weight modular forms and real quadratic $p$-rational fields}, {\sf Funct. Approx. Comment. Math.}, {\bf 63} (2020), 201-213.

\bibitem{ray}
R. Barbulescu and J. Ray, {\it Numerical verification of the Cohen-Lenstra-Martinet heuristics and of Greenberg's $p$-rationality conjecture}, {\sf J. Th\'{e}or. Nombres Bordeaux}, {\bf 32} (2020), 159-177.

\bibitem{jpaa}
Y. Benmerieme and A. Movahhedi, {\it Multi-quadratic $p$-rational number fields}, {\sf J. Pure Appl. Algebra}, {\bf 225} (2021), 17 pp.

\bibitem{self-acta}
J. Chattopadhyay and S. Muthukrishnan, {\it On the simultaneous $3$-divisibility of class numbers of triples of imaginary quadratic fields}, {\sf Acta Arith.}, {\bf 197} (2021), 105-110.

\bibitem{self-rama}
J. Chattopadhyay and A. Saikia, {\it Simultaneous indivisibility of class numbers of pairs of real quadratic fields}, {\sf Ramanujan J.}, {\bf 58} (2022), 905-911.

\bibitem{imrn}
G. Cherubini, A. Fazzari, A. Granville, V. Kala and P. Yatsyna, {\it Consecutive Real Quadratic Fields with Large Class Numbers}, {\sf Int. Math. Res. Not. IMRN}, (2022) https://doi.org/10.1093/imrn/rnac176.

\bibitem{elliot}
P.T.D.A. Elliott, {\it Probabilistic Number Theory I: Mean-Value Theorems}, {\sf Springer-Verlag}, (1979).

\bibitem{gras1}
G. Gras, {\it Class Field Theory, from Theory to Practice}, {\sf Springer Monographs in Mathematics, Springer-Verlag}, (2003).

\bibitem{gras2}
G. Gras, {\it Les $\theta$-r\'{e}gulateurs locaux d’un nombre alg\'{e}brique: conjectures $p$-adiques}, {\sf Can. J. Math.}, {\bf 68} (2016), 571-624.

\bibitem{gras3}
G. Gras, {\it On $p$-rationality of number fields. Applications – PARI/GP programs}, {\sf Publ. Math. Be-sancon Alg\'{e}bre Th\'{e}orie Nr.}, {\bf 2} (2019), 29-51.

\bibitem{greenberg}
R. Greenberg, {\it Galois representation with open image}, {\sf Ann. Math. Qu\'{e}.}, {\bf 40} (2016), 83-119.

\bibitem{heath-brown}
D. R. Heath-Brown, {\it Power-free values of polynomials.}, {\sf Q. J. Math.}, {\bf 64} (2013), 177-188.

\bibitem{azizul}
A. Hoque, {\it On a conjecture of Iizuka}, {\sf J. Number Theory}, {\bf 238}, (2022), 464-473.

\bibitem{iizuka}
Y. Iizuka, {\it On the class number divisibility of pairs of imaginary quadratic fields}, {\sf J. Number Theory}, {\bf 184} (2018), 122-127.

\bibitem{jaulent}
J.-F. Jaulent and T. Nguyen Quang Do, {\it Corps $p$-rationnels, corps $p$-r\'{e}guliers, et ramiﬁcation restreinte}, {\sf J. Th\'{e}or. Nombres. Bordeaux.}, {\bf 5} (1993), 343-363.

%

\bibitem{kop}
J. Koperecz, {\it Triquadratic $p$-rational fields}, {\sf J. Number Theory}, (2022) https://doi.org/10.1016/j.jnt.2022.04.011.

\bibitem{tvm}
S. Krishnamoorthy and S. K. Pasupulati, {\it Note on the $p$-divisibility of class numbers of an infinite family of imaginary quadratic fields}, {\sf Glasg. Math. J.}, {\bf 64} (2022), 352-357.

\bibitem{le}
M. H. Le, {\it Upper bounds for class numbers of real quadratic fields}, {\sf Acta Arith.}, {\bf 68}, (1994), 141-144.

\bibitem{loub}
S. Louboutin, {\it The Brauer-Siegel theorem}, {\sf J. Lond. Math. Soc.}, {\bf 72} (2005) 40-52.

\bibitem{mova-thesis}
A. Movahhedi, {\it Sur les $p$-extensions des corps $p$-rationnels}, {\sf PhD Thesis}, (1988).

\bibitem{mova2}
A. Movahhedi, {\it Sur les $p$-extensions des corps $p$-rationnels}, {\sf Math. Nachr.}, {\bf 149} (1990), 163-176.

\bibitem{mova3}
A. Movahhedi and T. Nguyen Quang Do, {\it Sur l’arithm\'{e}tique des corps de nombres $p$-rationnels, in:
Séminaire de Th\'{e}orie des Nombres, Paris 1987-1988, in: Progress in Mathematics}, {\bf 81}, Birkh\"{a}user, Boston Inc., (1990).

\bibitem{ram}
R. M. Murty and J. Esmonde, {\it Problems in Algebraic Number Theory, 2nd ed.}, {\sf Graduate Texts in Mathematics}, {\bf 190}, Springer-Verlag, New York (2005). 

\bibitem{serre}
J.-P. Serre, {\it Corps locaux}, {\sf Publications de l'Institut de Math\'{e}matique de l'Universit\'{e} de Nancago, VIII Actualit\'{e}s Sci. Indust.}, {\bf 1296}, Hermann, Paris (1962), 243 pp.

\bibitem{rama-simul}
J.-F. Xie and K. F. Chao, {\it On the divisibility of class numbers of imaginary quadratic fields $\mathbb{Q}(\sqrt{D})$, $\mathbb{Q}(\sqrt{D + m})$}, {\sf Ramanujan J.}, {\bf 53} (2020), 517-528.




\end{thebibliography}
\end{document}